\documentclass[11pt]{article}
\usepackage{amsmath,authblk}
\usepackage{mathrsfs}
\usepackage{amsfonts}

\usepackage{amsfonts, amsmath, amssymb}
\usepackage{amssymb,amsfonts,amsmath,
latexsym, epsfig,cite, psfrag,eepic,colordvi,color}
\usepackage{amscd,graphics}
\usepackage{lineno}
\allowdisplaybreaks

\newcommand{\qed}{\hfill $\Box $}

\newtheorem{theorem}{Theorem}[section]
\newtheorem{lemma}[theorem]{Lemma}

\allowdisplaybreaks[4]
\usepackage{caption}
\begin{document}

\title{Stability  for the anti-Ramsey number of matchings\thanks{Research supported by National Key Research and Development Program of China 2023YFA1010203, National Natural Science Foundation of China under grant 12271425}} 

\author{Xuechun Zhang \thanks{Corresponding email: xjtu\_xczhang@163.com (X. Zhang)}
  and Hongliang Lu\\
\small School of Mathematics and Statistics\\
Xi'an Jiaotong University \\ \small  Xi'an, Shaanxi 710049, P.R.China }

\date{}

\maketitle

\date{}

\begin{abstract}
Let $n, r, s$ be three positive integers such that $n\geq 2s+5$. Let $K_r$ denote the complete graph of order $r$. Given a graph $F$, the anti-Ramsey number $ar(n,F)$ is defined as  the minimum number $C$ such that any edge-coloring of  $K_n$ with exactly $C$ colors contains a
rainbow copy of $F$.
Let  $H$ be an edge-colored graph on $K_n$ with at least $g(n,s)$ colors, where
\[
 g(n,s)=\max\left\{ \binom{n}{2} - \binom{n - s + 1}{2} + 5, \binom{2s - 1}{2} + n + 1 \right\}.
\] 
In this paper, we establish a stability type result for the anti-Ramsey number of matchings. Specifically, if  $H$
 does not have a rainbow matching of size $s+2$, then $H$ contains either a monochromatic complete graph $K_{n-s}$ or a monochromatic \( K_{n - 2s - 1} \vee \overline{K_{2s + 1}} \).
\end{abstract}


\section{Introduction}
We begin with some standard definitions.  Consider a simple graph $G$ with vertex set $V(G)$ and edge set $E(G)$. The \emph{order} and \emph{size} of $G$ are denoted by $n = |V(G)|$ and $e(G) = |E(G)|$, respectively. Given an edge subset $E'$, we use  $G - E'$ to represent the graph whose vertex set is  $V(G)$, and whose edge set is  $E(G)-E'$.  When $E'=\{e\}$, we use $G-e$ to denote $G-E'$. Given a vertex subset $S\subseteq V(G)$, the graph $G - S$ is defined as the subgraph of $G$ whose vertex set is $V(G)-S$, and whose edge set consists of all the edges in $E(G)$ that do not have any endpoints in $S$. If $S=\{v\}$, we use $G-v$ to denote $G-S$. The \emph{degree} of a vertex $v$ in $G$, denoted by $d_G(v)$, is the number of edges
of $G$ incident with $v$. For a vertex set $S$, the \emph{neighborhood} $N_G(S)$ is the set of vertices adjacent to some vertex in $S$; that is, $N_G(S) = \{ x \mid xy \in E(G) \text{ and } y \in S \}$. For vertex subsets $S$ and $T$, let $E_G(S,T)$ denote the set of edges with one end in $S$ and the other in $T$. Let $e_G(S,T): = \lvert E_G(S,T) \rvert$.
The maximal connected subgraphs of $G$ are its \emph{connected components}.
The complete graph on $p$ vertices is denoted by $K_p$.  A \emph{matching} in $G$ is a set of edges no two of which are incident, and the \emph{matching number} $\nu(G)$ is the size of a largest matching in $G$. A matching of size $k$ is denoted by $M_k$.
Let $G_1 = (V_1, E_1)$ and $G_2 = (V_2, E_2)$ be two graphs such that $V_1 \cap V_2 = \emptyset$. The \emph{union} of graphs $G_1$ and $G_2$, denoted by $G_1 \cup G_2$, is the graph with vertex set $V_1 \cup V_2$ and edge set $E_1 \cup E_2$. The \emph{join} of graphs $G_1$ and $G_2$, denoted by $G_1 \vee G_2$, is the graph union $G_1 \cup G_2$ together with all the edges joining $V_1$ and $V_2$. Given two integers $n_1$  and $n_2$  with
$n_1\leq n_2$, we use the notation $[n_1,n_2]$ to denote the set $\{n_1,\ldots,n_2\}$. For the sake of simplicity, when $n_1=1$, we write $[n_2]$ instead of $[n_1,n_2]$.


A graph $F$ is \emph{rainbow} if all edges have distinct colors. The \emph{anti-Ramsey number} $ar(n, F)$ is  the smallest integer $m$ such that each edge-coloring of  $K_n$ with exactly $m$ colors contains a rainbow copy of $F$.  A \emph{rainbow matching} is a matching where all edges have distinct colors. 

For a family of graphs \( \mathcal{F} \), the \emph{Tur\'an number} $ex(n, \mathcal{F})$ is the maximum possible number of edges in an $n$-vertex graph which does not contain any $F\in \mathcal{F}$ as a graph. When $\mathcal{F} = \{F\}$, we write $ex(n, F)$ for $ex(n, \mathcal{F})$. The value of $ar(n, G)$ is closely related to the Tur\'an number $ex(n, \mathcal{G})$, where $\mathcal{G} = \{G - e : e \in E(G)\}$, via the following inequality \cite{ESS}:
\[
2 + ex(n, \mathcal{G}) \leq ar(n, G) \leq 1 + ex(n, G).
\]
The upper bound follows directly from the definition of the Tur\'an number. For the lower bound, consider a $\mathcal{G}$-free graph $H$ with $ex(n, \mathcal{G})$ edges. Color the edges of $H$ with distinct colors, and use one additional color for all edges in its complement $\overline{H}$. This colored graph with $1 + ex(n, \mathcal{G})$ colors and  no rainbow copy of $G$.

Erd\H{o}s, Simonovits, and S\'os~\cite{ESS} proved that for any fixed integer $l\geq 2$ and sufficiently large $n$, the anti-Ramsey number is $ar(n, K_{l+1}) = ex(n,K_l)+2$. This result was later refined by Montellano-Ballesteros and Neumann-Lara \cite{MN}, who showed that it holds for all $n\geq l+1$. Schiermeyer~\cite{S} was the first to apply the counting technique to the matching problem, proving that \( ar(n, M_s) = ex(n, M_{s-1}) + 2 \) for \( s \geq 2 \) and \( n \geq 3s + 3 \). This was later improved by Fujita et al. \cite{FKSS}, who extended the range to \( n \geq 2s + 1 \). The problem was finally resolved by Chen, Li, and Tu \cite{CLT} for all \( s \geq 2 \) and \( n \geq 2s \), giving the exact value of \( ar(n, M_s) \). Independently, Haas and Young \cite{HY} obtained the same result for the case \( n = 2s \) by a simpler argument. In the present paper, our objective is to establish a stability result concerning the anti-Ramsey number of matchings.
 \begin{theorem}\label{ARM}
 Let \( n \) and \( s \) be positive integers satisfying \( n \geq \max \{2s + 5, 40\} \). Let \( H \) be an edge-colored graph  $K_n$ with at least $g(n,s)$ colors, where $g(n,s)$   is defined as
 \begin{align}\label{extremal graph}
 g(n,s)=\max\left\{ \binom{n}{2} - \binom{n - s + 1}{2} + 5, \binom{2s-1}{2} + n + 1 \right\}.
\end{align}
 If \( H \) contains no rainbow matching of size \( s + 2 \),  then \( H \) contains either a monochromatic \( K_{n - s } \) or a monochromatic \( K_{n - 2s - 1} \vee \overline{K_{2s + 1}} \).
\end{theorem}

Define $g_1(n,s):=\binom{n}{2} - \binom{n - s + 1}{2} + 5$ and $g_2(n,s):=\binom{2s - 1}{2} + n + 1$. 
 Observe that the function $g(n,s)$
 takes the value of $g_1(n,s)$
 when the variable $n$
 satisfies the inequality $n\geq \frac{5}{2}s+\frac{3}{2}$,  while it takes the value of
  $g_2(n,s)$ when  when
$n$  lies in the interval $2s+5\leq n\leq \frac{5}{2}s+\frac{3}{2}$. 

The bounds in Theorem \ref{ARM} are tight, which we verify by constructing explicit edge-colorings of \(K_n\) as follows:

First, let \(n \geq \frac{5}{2}s + \frac{3}{2}\). Define the subgraph $F_1$ of $ K_n$ as \(F_1 := K_{s-1} \vee (K_3 \cup \overline{K_{n-s-2}})\), with edge set \(E(F_1) = \{e_1, \ldots, e_{g_1(n,s)-2}\}\). We define an edge-coloring \(f_1: E(K_n) \to [g_1(n,s)-2] \cup \{0\}\) and denote the colored graph as \(H_1\), where:
\[
f_1(e) =
\begin{cases}
i & \text{if } e = e_i \in E(F_1), \\
0 & \text{otherwise}.
\end{cases}
\]
It is straightforward that \(H_1\) contains no rainbow matchings of size \(s+2\) and no monochromatic \(K_{n-s}\).

Next, let \(2s+5 \leq n \leq \frac{5}{2}s + \frac{3}{2}\). Define the subgraph $F_2$ of $K_n$ as \(F_2 := K_1 \vee (K_{2s-1} \cup \overline{K_{n-2s}})\), with edge set \(E(F_2) = \{e_1, \ldots, e_{g_2(n,s)-2}\}\). We define an edge-coloring \(f_2: E(K_n) \to [g_2(n,s)-2] \cup \{0\}\) and denote the colored graph as \(H_2\), where:
\[
f_2(e) =
\begin{cases}
i & \text{if } e = e_i \in E(F_2), \\
0 & \text{otherwise}.
\end{cases}
\]
It is straightforward that \(H_2\) contains no rainbow matchings of size \(s+2\) and no monochromatic \(K_{n-2s-1} \vee \overline{K_{2s+1}}\).

The rest of this paper is organized as follows. In Section 2, we present some preliminary technical lemmas required for our proof. The proof of Theorem \ref{ARM} is provided in Section 3.

 \section{Preliminaries}
 We now present the preliminary results required for the proof of our main result. First, we need the classical result of Erd\H{o}s and Gallai~\cite{EG} on the Tur\'an number of matchings.
 \begin{theorem}[Erd\H{o}s and Gallai, \cite{EG}]\label{Matching}
 	For $n \geq 2s+2$ and $s \geq 1$,
 \begin{align}
 	ex(n,M_{s+1})= \max \left\{e(G(n,s)), {2s+1\choose 2}\right\},
 \end{align}
 where $G(n, s)$ is the graph $K_s \vee \overline{K_{n-s}}$, consisting of a clique on $s$ vertices, each connected to all vertices of an independent set on $n-s$ vertices.
\end{theorem}


We need the following result on anti-Ramsey number of matching.
\begin{theorem}[\cite{S,FKSS,CLT}]\label{anti-matching}
	\begin{align}
		ar(n, M_s) = \begin{cases}
			4, & \text{if } n = 4 \text{ and } s = 2, \\
			ex(n, M_{s-1}) + 3, & \text{if } n = 2s \text{ and } s \geq 7, \\
			ex(n, M_{s-1}) + 2, & \text{otherwise}.
		\end{cases}
	\end{align}
\end{theorem}

A graph $G$ is \emph{factor-critical} if $G-v$ has a perfect matching for every vertex $v$ of $G$. Note that if $G$ is factor-critical, then it has odd order.
We need to use the strengthened version of Berge's formula, which can be obtained from Berge Formula and Gallai-Edmonds Structure Theorem\cite{B, YL}.

\begin{theorem}[Berge~\cite{B}]\label{BFormula}
	For any graph $G$, the matching number is given by
	\begin{align}\label{Berge}
		\nu(G) = \frac{1}{2}\left(|V(G)| - \max_{S \subseteq V(G)} \{odd(G-S) - |S|\}\right),
	\end{align}
	where $odd(G-S)$ denotes the number of factor-critical components in $G-S$.
\end{theorem}

A fundamental result characterizing the existence of perfect matchings in bipartite graphs is provided by Hall's Marriage Theorem.
\begin{theorem}[Hall~\cite{H}]\label{Hall}
	 Let $G$ be a bipartite graph with bipartition $
	V(G)=A \cup B$  and such that $|N_G(S)|\geq|S|$ for all $S\subseteq A$. Then $G$ has a matching covering $A$.
\end{theorem}

By Theorem \ref{Hall}, we can derive the following lemma.
\begin{lemma}\label{lemma}
	Let $G$ be a bipartite graph with bipartition $V (G) =
	A \cup B$, where $A=[a]$ and $d_G(i)\geq i$ for all $i\in A$. Then $G$ has a matching covering $A$.
\end{lemma}

\noindent\textbf{Proof:}
	Let $S\subseteq A$, and let $k=\max\{i\mid i\in S\}$; so $|S|\leq k$. Since $d_G(k)\geq k$, and $k\in S$, we have
	\[
	|N_G(S)|\geq d_G(k)\geq k\geq |S|.
	\]
	So by Hall's Theorem,  $G$ has a matching covering $A$.\qed


\section{Proof of Theorem \ref{ARM} }

Let $f:E(K_n)\rightarrow [g(n,s)]$ be a surjective edge-coloring of complete graph $K_n$. We denote the edge-colored graph by $H$.
Write $V(H):=\{v_1,v_2,\ldots,v_n\}$, $e_{ij}=\{v_i, v_j\}$ and $E(H):= \{e_{ij}\ |\  1\leq i < j \leq n\}$ .     For $1 \leq i \leq g(n,s)$, let $\mathcal{F}_i$ be the set of edges in $H$ that have color $i$. Consider a rainbow subgraph $G$ of $H$ on vertex set $V(H)$ with  exactly $g(n,s)$ edges,  i.e., one satisfying  $|E(G)\cap \mathcal{F}_i|=1$ for each $i\in [g(n,s)]$.

Note that \begin{align*}
	g(n,s)&\geq\max\left\{ \binom{n}{2} - \binom{n - s+1}{2} + 5, \binom{2s - 1}{2} + n+1  \right\} \\
	&>\max\left\{ \binom{n}{2} - \binom{n - s+1}{2} + 2, \binom{2s - 1}{2} +2 \right\} \\
	&=ex(n, M_s) + 2.
\end{align*}
 By  Theorem \ref{anti-matching}, we can see that $e(G)=  g(n,s) > ar(n, M_{s+1})$. Since $H$ contains no rainbow matching of size $s+2$, we may assume that $\nu(G)=s+1$ .  By Theorem \ref{BFormula}, there exists $T\subseteq V(G)$ such that
\begin{align}\label{BB}
\nu(G)=\frac{1}{2}(|V(G)|-(odd(G-T)-|T|))
\end{align}
and hence
\begin{align}
	n-2(s+1)=odd(G-T)-|T|.
\end{align}
 Let $q = odd(G-T)$, $t = |T|$, and let $C_1, \ldots, C_q$ be the factor-critical  components of $G-T$, where $|V(C_i)| = 2k_i + 1$ and $k_1 \geq k_2 \geq \cdots \geq k_q$.
It follows that $G$ is a subgraph of $Q := K_t \vee \left( \bigcup_{i=1}^q K_{2k_i+1} \right)$ and
 \[
 t + k_1 + k_2 + \cdots + k_q = s+1.
 \]
Specifically, if $|k_i|=1$, then $C_i$ is isomorphic to $K_3$.

\medskip
\textbf{Claim 1.} For fixed $t$, we have $e(G) \leq e\left(K_t \vee \left(K_{2s+3-2t} \cup \overline{K_{n-2s-3+t}}\right)\right)$.
\medskip

Suppose $k_2 \geq 1$. Let $Q' = K_t \vee \left(K_{2k_1+3} \cup K_{2k_2-1} \cup \left(\bigcup_{i=3}^q K_{2k_i+1}\right)\right)$. Then
\begin{align*}
	e(Q') = e(Q) - (4k_2 - 1) + (2k_1 + 1) + (2k_1 + 2) = e(Q) - 4k_2 + 4k_1 + 4 > e(Q).
\end{align*}
Repeating this process yields the inequality
\begin{align*}
	e(G) \leq e(Q) \leq e\left(K_t \vee \left(K_{2s+3-2t} \cup \overline{K_{n-2s-3+t}}\right)\right).
\end{align*}
This completes the proof of Claim 1.

Define
\begin{align*}
	f(n,t) := \binom{n}{2} - \binom{n-t}{2} + \binom{2s+3-2t}{2}.
\end{align*}
By Claim 1, we deduce that $e(G) \leq f(n,t)$. The function
\begin{align*}
	f(n,t) = \frac{3}{2}t^2 + \left(n - 4s - \frac{11}{2}\right)t + 2s^2 + 5s + 3
\end{align*}
is convex for $t \in [0, s+1]$.

Let   $T = \{v_1,v_2,\ldots,v_t\}$ when $t \geq 1$. Suppose that $C_1, C_2, \ldots, C_q$ have vertex sets $\{v_{t+1}, \ldots, v_{2k_1 + t + 1}\}$, $\{v_{2k_1 + t + 2}, \ldots, v_{2k_1 + 2k_2 + t + 2}\}$, $\ldots$, $\{v_{n - 2k_q}, \ldots,v_n\}$, respectively. Without loss of generality, assume the vertices in $T$ are sorted so that
\begin{align}\label{order}
	e_G(v_1, G-T) \leq e_G(v_2, G-T) \leq \cdots \leq e_G(v_t, G-T).
\end{align}

\medskip
\textbf{Claim 2.} $t\in\{1, 2, 3, s-1, s, s+1\}$.

By contradiction, suppose that the result does not hold. Firstly, consider $t=0$. We claim that $G$ is a subgraph of $K_{2s+3} \cup \overline{K_{n-2s-3}}$ or $K_{2s+1} \cup K_3 \cup \overline{K_{n-2s-4}}$. Otherwise, we have
\begin{align*}
	e(G) &\leq e(K_{2s-1} \cup K_5 \cup \overline{K_{n-2s-4}}) \\
	&= \binom{2s-1}{2} + 10 \\
	&< g_2(n, s),
\end{align*}
a contradiction. Now, let $e\in \{e_{ij}\in E(\overline{G})\ | \ 2s+5\leq i< j\leq  n\}$. Then there exists an edge $e' \in E(G)$ with $f(e') = f(e)$. Since all components of $G$ are factor-critical, $G-e'$ has a  matching $M$ of size $s+1$ in $G - e'$. Consequently, $M \cup \{e\}$ is a rainbow matching of size $s+2$ in $H$, a contradiction.

Next we assume that $4 \leq t \leq s-2$. If $n \geq \max \{\frac{5}{2}s + \frac{3}{2}, 40\}$, then
	\begin{align*}
		f(n,t) &\leq \max \{f(n,4), f(n,s-2)\} \\
		&= \max\left\{4n+2s^2-11s+5,\ ns-2n-\tfrac{1}{2}s^2+\tfrac{3}{2}s+20\right\}\\
		&< g_1(n,s).
	\end{align*}
If $40\leq n < \frac{5}{2}s + \frac{3}{2}$, then $s>15$ and
	\begin{align*}
		f(n,t) &\leq \max \{f(n,4), f(n,s-2)\} \\
		&= \max\left\{4n+2s^2-11s+5,\ ns-2n-\tfrac{1}{2}s^2+\tfrac{3}{2}s+20\right\}\\
		&= 4n+2s^2-11s+5 \\
		&< g_2(n,s).
	\end{align*}
Thus, for $4 \leq t \leq s-2$, we have
\[
e(G)\leq \max\{f(n,4), f(n,s-2)\} < \max\left\{g_1(n,s), g_2(n,s)\right\},
\]
a contradiction.

\medskip
\textbf{Claim 3.~}If $t\in\{1,2,3,s,s+1\}$, then $G$ is a subgraph of $K_t\vee (K_{2s-2t+3}\cup \overline{K_{n-2s+t-3}})$; else if $t=s-1$, then $G$ is a subgraph of $K_{s-1}\vee (K_5\cup\overline{ K_{n-s-4}})$ or $G\subseteq K_{s-1}\vee (2K_3\cup \overline{K_{n-s-5}})$.

If $t\in \{s+1, s\}$, then by Theorem \ref{BFormula}, $k_2=0$ and $k_1=s+1-t$; so the result is followed. Consider $t=s-1$. By Theorem \ref{BFormula}, we have $\sum_{i\in [q]}k_i=2$.
One can see that $k_1=k_2=1$ and $k_3=0$ or $k_1=2$ and $k_2=0$. Thus we may infer that $G$ is a subgraph of $K_{s-1}\vee (K_5\cup\overline{ K_{n-s-4}})$ or $ K_{s-1}\vee (2K_3\cup \overline{K_{n-s-5}})$.

Next we may assume that $t\in\{1,2,3\}$, $t\leq s-2$ and $s\geq 4$. Suppose that $G$ is not a subgraph of $K_t\vee (K_{2s-2t+3}\cup \overline{K_{n-2s+t-3}})$. Then we have $k_2\geq 1$.

Firstly, consider $t=1$.  If $k_2\geq 2$ or $k_3\geq 1$, then we have
\begin{align*}
e(G)&\leq e\left(K_1 \vee (K_{2s-3} \cup K_5 \cup \overline{K_{n-2s-3}})\right)\\
&\leq {2s-3\choose 2}+n+9<g_2(n,s),
\end{align*}
a contradiction.
So  we may infer that $k_2=1$ and $k_3=0$. Then $G$ is a subgraph of $K_1 \vee (K_{2s-1} \cup K_3 \cup \overline{K_{n-2s-3}})$, and
\begin{align*}
	g(n,s)\leq e(G) &\leq e\left(K_1 \vee (K_{2s-1} \cup K_3 \cup \overline{K_{n-2s-3}})\right) \\
	&= \binom{2s-1}{2} + n + 2 \\
	&= g_2(n, s) + 1.
\end{align*}
Thus $G$ differs from  $K_1 \vee (K_{2s-1} \cup K_3 \cup \overline{K_{n-2s-3}})$ by at most one edge.
Let $e\in \{e_{ij}\in E(\overline{G})\ |\ 2s+4\leq i<j\leq n\}$. Then there exists an edge $e' \in E(G)$ such that $f(e') = f(e)$. 
Since the components of $G - T$ are factor-critical, it follows that there is a rainbow matching $M$ of size $s$ in $G-T-e'$. Moreover,  
\[
d_G(v_1)\geq g_2(n,s)-\binom{2s-1}{2}-3 \geq n-2.
\]
Therefore, there exists an edge $e''=\{v_1, v'\}$ such that $v'\notin V(M\cup \{e\})$ and $f(e'')\neq f(e)$.
Then $M \cup \{e,e''\}$ forms a rainbow matching of size $s+2$ in $H$, a contradiction.

Secondly, consider $t=2$.
By Theorem~\ref{BFormula}, $G$  is a subgraph of
\[ K_2 \vee \left( \bigcup_{i=1}^q K_{2k_i+1} \right) \quad \text{with} \quad \sum_{i=1}^{q} k_i = s-1. \]
If $k_2\geq 1$,  by following the same approach as in Claim~1,  we have
\begin{align*}
	e(G)&\leq e(K_2 \vee \left(K_{2s-3} \cup K_3\cup \overline{K_{n-2s-2}}\right))\\
	&=\binom{2s-3}{2} + 3 + 2(n-2) + 1 \\
	&= 2n + 2s^2 - 7s + 6 \\
	&< \max\{g_1(n,s), g_2(n,s)\},
\end{align*}
a contradiction.
Thus, we may assume that $k_2 = 0$. Then $G$ is a subgraph of $K_2 \vee \left(K_{2s-1} \cup \overline{K_{n-2s-1}}\right)$.



Finally, we assume that $t=3$ and $s\geq 5$. By Theorem~\ref{BFormula}, $G$ is a subgraph of
\[ K_3 \vee \left( \bigcup_{i=1}^q K_{2k_i+1} \right) \quad \text{with} \quad \sum_{i=1}^{q} k_i = s-2. \]
If $k_2\geq 1$, then we have
\begin{align*}
	e(G)&\leq e(K_3 \vee \left(K_{2s-5} \cup K_3\cup \overline{K_{n-2s-1}}\right))\\
	&=\binom{2s-5}{2} + 3 + 3(n-3) + 3 \\
	&= 3n + 2s^2 - 11s + 12 \\
	&< \max\{g_1(n,s), g_2(n,s)\},
\end{align*}
a contradiction.

Next we discuss the following five cases.

\medskip
\textbf{Case 1.} $t=1$.

According to Claim 3, 
$G$  is a subgraph of the graph
\[
K_1 \vee (K_{2s+1}\cup\overline{K_{n-2s-2}}).
\]

\textbf{Subcase 1.1.} $|\{f(e_{ij})\ |\ 2s+3\leq i<j\leq n\}|\geq 2$.

%

First, let $n\geq 2s+7$. There exist two vertex-disjoint edges $e_1, e_2$ with $V(\{e_1, e_2\})\subset \{v_{2s+3},\ldots, v_n\}$ such that $f(e_1)\neq f(e_2)$. Let $V_1:=\{v_1,v_2,\ldots,v_{2s+2}\}$. Then $G[V_1]$ contains at least
\[
g_2(n,s) - (n - 2s - 2) - 2 = 2s^2 - s + 2
\]
edges $h$ such that $f(h)\notin \{f(e_1),f(e_2)\}$.  Note that $ex(2s+2, M_s) = \binom{2s-1}{2} < 2s^2 - s + 2$. By Theorem~\ref{Matching}, there exists a rainbow matching $M$ of size $s$ in $G[V_1]$ such that for any $h\in M$, $f(h)\notin \{f(e_1),f(e_2)\}$. Then $M\cup \{e_1, e_2\}$ is a rainbow matching of size $s+2$ in $H$, a contradiction.

Next, let $n=2s+6$. There exist two vertex-disjoint edges $e_1, e_2$ with $v_1\in V(\{e_1\})$ and $V(\{e_1, e_2\})\subset \{v_1, v_{2s+3}, v_{2s+4}, v_{2s+5}, v_{2s+6}\}$ such that $f(e_1)\neq f(e_2)$. Then $G[V_1]-v_1$ contains at least
\[
g_2(n,s)-(n-1)-1=\binom{2s-1}{2}+1
\]
edges $h$ such that $f(h)\notin \{f(e_1), f(e_2)\}$. Since $ex(2s+1, M_s)=\binom{2s-1}{2}<\binom{2s-1}{2}+1$, by Theorem 2.1, there exists a rainbow matching $M$ of size $s$ in $G[V_1]-v_1$ such that  for any $h\in M$, $f(h)\notin \{f(e_1), f(e_2)\}$. Then $M\cup \{e_1, e_2\}$ is a rainbow matching of size $s+2$ in $H$, a contradiction.

Now consider $n=2s+5$. Without loss generality, suppose that $f(e_{2s+3, 2s+4})=c$ and $f(e_{2s+4, 2s+5})=c'$. 
Define 
\[
G':=\left\{
     \begin{array}{ll}
       G[V_1]-e, & \hbox{if there exists an edge \(e\in E(G[V_1])\) such that \(f(e)=c\);} \\
       G[V_1], & \hbox{otherwise.}
     \end{array}
   \right.
\]
 Since there is no rainbow matching of size $s+2$ in $H$, it follows that $\nu(G')\leq s$. If $G'$ contains no isolated vertices, then we have
\begin{align*}
e(G)&\leq e(G')+3+1\\
&\leq e(K_1\vee (K_{2s-1}\cup 2K_1))+4\\
&\leq {2s\choose 2}+2+4\\
&={2s-1\choose 2}+n\quad\mbox{(since $n= 2s+5$)},
\end{align*}
contradicting $e(G)\geq g_2(n, s)$. So we may assume that
$G'$ contains an isolated vertex $v_x$. For any   $v_i\in V_1- v_x$, we claim $f(e_{xi})=c$. Otherwise, suppose that
 there exists $v_y\in V_1-v_x$ such that  $f(e_{xy})\neq c$.  Then $G[V_1]-v_x-v_y$ contains at least
\begin{align}\label{sub-case1-eq1}
g_2(n,s)-3-2s-2=\binom{2s-1}{2}+1>ex(2s, M_s)
\end{align}
edges $h$ such that $f(h)\notin \{c, f(e_{xy})\}$. 
By Theorem \ref{Matching}, $G[V_1]-v_x-v_y$  contains a rainbow matching $M$ of size $s$ such that for any $h\in M$, $f(h)\notin \{c, f(e_{xy})\}$. 
Then $M\cup \{e_{2s+3, 2s+4}, e_{xy}\}$ is a rainbow matching of size $s+2$ in $H$, a contradiction.  
Fix $v_i\in V_1-v_x$. Then $\{e_{xi}, e_{2s+4, 2s+5}\}$ is a rainbow matching of size 2. By inequality (\ref{sub-case1-eq1}) and  Theorem \ref{Matching}, $G[V_1]-v_x-v_i$ 
contains a rainbow matching $M$ of size $s$ in $G[V_1]-v_x-v_i$ such that for any $h\in M$, $f(h)\notin \{c, c'\}$. Then $M\cup \{e_{2s+4, 2s+5}, e_{xi}\}$ is a rainbow matching of size $s+2$ in $H$, a contradiction again.

\medskip
\textbf{Subcase 1.2.} $|\{f(e_{ij})\ |\ 2s+3\leq i<j\leq n\}|=1$.


Let $f(e_{ij})=c$ for all $\{i,j\}\subset [2s+3, n]$. Similarly, we define $G'$ as follows: 
\[
G'=\left\{
     \begin{array}{ll}
       G[V_1]-e, & \hbox{if there exists an edge \(e\in E(G[V_1])\) such that \(f(e)=c\);} \\
       G[V_1], & \hbox{otherwise.}
     \end{array}
   \right.
\]
  Since $H$ contains no rainbow matching of size $s+2$, we know that  $\nu(G')\leq s$.
If $G'$ contains no isolated vertices, then we have
\begin{align*}
e(G)&\leq e(G')+(n-2s-2)+1\\
&\leq e(K_1\vee (K_{2s-1}\cup 2K_1))+(n-2s-2)+1\\
&\leq {2s-1\choose 2}+n,
\end{align*}
contradicting $e(G)\geq g_2(n, s)$. So we may assume that
 $G'$ contains an isolated vertex $v_x$.

  We claim that $G'-v_x$ is factor-critical. Otherwise, there exists $v_y\in V(G')-v_x$ such that $G'-v_x-v_y$ contains no perfect matchings. Then 
\begin{align*}
e(G)&\leq e(G')+1+(n-2s-2)\\
&=e(G'-v_x)+1+(n-2s-2)\\
&\leq e(G'-v_y)+2s+1+(n-2s-2)\\
&< {2s-1\choose 2}+n,\quad \mbox{(by Theorem~\ref{Matching})}\\
\end{align*}
a contradiction.

We firstly show that $H-(V(G')-v_x)$ is a monochromatic graph  in $H$. Suppose, for the sake of contradiction, that $H-(V(G')-v_x)$ is not monochromatic. Then there exist  two vertex-disjoint edges $e_1,e_2$ in $H-(V(G')-v_x)$   such that $f(e_1)\neq f(e_2)$ and $f(e_1)=c$. Since $G'-v_x$ is factor-critical, $G'-v_x$ has a matching $M$  of size $s$ avoiding color $f(e_2)$. Then $M\cup \{e_1,e_2\}$ is a matching of size $s+2$, a contradiction.

Next, it suffices to show that \(f(e_{uv}) = c\) for all edges \(e_{uv}\) with \(u \in [2s+2] - \{x\}\) and \(v \in [2s+3, n] \cup \{x\}\). Suppose, for the sake of contradiction, that this is not the case.
Then there exist $u\in [2s+2]- \{x\}$ and $v\in [2s+3, n ]\cup \{x\}$ such that  $f(e_{uv}) \neq c$. Let $e_3\in E(G)$ such that $f(e_3) = f(e_{uv})$. We select  an edge $e'=e_{ij}$ with $\{i,j\}\subseteq ([2s+3, n]\cup \{x\}) - \{v\} $, $f(e')= c$, and an edge $e'' \in E(G)$ with  $f(e'') = c$. Since  $v_x$ is an isolated vertex in $G'$ and $H-(V(G')-v_x)$ is  monochromatic, we have
\begin{align*}
&e(G[V_1] - \{v_x,v_u\}-\{ e_3, e''\}) \\
&\geq e(G[V_1])-2-2s\\
&\geq g_2(n,s)-(n-2s-2) - 2s - 2 \\
&> \binom{2s-1}{2}.
\end{align*}
Hence, by Theorem~\ref{Matching}, 
$G[V_1] -\{v_x,v_u\} -\{ e_3, e''\}$ has a  matching $M$ of size $s$. Then $M \cup \{e_{uv}, e'\}$ forms a rainbow matching of size $s+2$ in $H$, a contradiction.
This completes the proof of Case 1.

\medskip
\textbf{Case 2.} $t = s-1$.

According to Claim 3, 
$G$  is  a subgraph of either $K_{s-1} \vee (K_5 \cup \overline{K_{n-s-4}})$ or $K_{s-1} \vee (2K_3 \cup \overline{K_{n-s-5}})$.
 Let $e:= e_{n-1,n}$. Then there exists an edge $e' \in E(G)$ such that $f(e') = f(e)$. Since the components of $G - T$ are factor-critical, we have $\nu(G - T - e') = 2$. 
Let $M$ be a matching of size $2$ in $G - T - e'$. Let $G'$ be a bipartite graph with bipartition $(T,B)$ and edge set $E_{G-e'}(T,B)$,  where
 $B:=\{v_{s},v_{s+1},\ldots, v_{n-2}\}-V(M)$.
Then  for any $v_i\in T$, we know that
\begin{align}\label{diff1}
	e_G(\{v_i\}, V(G-T)) \leq e_{G'}(v_i, B) + 7.
\end{align}

\medskip
\textbf{Claim 4.~}$e_{G'}(\{v_i\}, B)\geq i$ for all $v_i\in T$.

Suppose for contradiction that there exists $v_j\in T$ such that $e_{G'}(\{v_j\}, B)\leq j - 1$. By (\ref{order}) and (\ref{diff1}), one can see that  $e_{G}(v_k, V(G-T))\leq j + 6$ 
for all $k\in [j]$. So we have
\begin{align*}
	\sum_{k=1}^je_G(\{v_k\}, V(G-T)) &\leq j(j+6).
\end{align*}
Note that $e_G(\{v_{j+1},\ldots,v_{s-1}\}, V(G-T))\leq (s-1 - j)(n- s + 1)$ and   $e(G-T)\leq 10$. It follows that
\begin{align*}
	e(G)\leq \binom{s-1}{2} + (s - 1- j)(n-s + 1) + j(j+6) + 10
	:=h_1(n, s, j).
\end{align*}
The function $h_1(n, s, j)$ is convex for $j$ in the interval $[s-1]$.  So we have $h_1(n, s, j)\leq \max\{h_1(n, s, 1), h_1(n, s, s-1)\}$. By $n\geq \max \{2s+5, 40\}$, we have
\begin{align*}
	h_1(n, s, 1)=\binom{s-1}{2} + (s - 2)(n-s+1) + 17
	<g_1(n,s).
\end{align*}
For $j=s-1$, we have
\begin{align*}
	h_1(n, s, s-1)&=\binom{s-1}{2} + (s-1)(s+5) + 10\\
	&=\frac{3}{2}s^2 + \frac{5}{2}s + 6<g_2(n,s).
\end{align*}
Thus,
$e(G)\leq  \max \{h_1(n, s, 1), h_1(n, s, s-1)\}<  \max \{g_1(n,s), g_2(n,s)\}$, a contradiction. This completes the proof of Claim 4.

  According to Lemma \ref{lemma} and Claim 4, $G'$ contains a matching $M'$ of size $s-1$. Then $M' \cup M\cup \{e\}$ forms a rainbow matching of size $s+2$ in $H$,  a contradiction. This completes the proof of Case 2.

%
%
%

\medskip
\textbf{Case 3.} $t = s+1$.

According to Claim 3, 
$G$  is a subgraph of $K_{s+1} \vee \overline{K_{n-s-1}}.$ Let $V_2:=\{v_1,v_2\}\cup \{v_{s+2},v_{s+3}$, $\ldots, v_n\}$ and $M$ be a maximum rainbow matching in $H[V_2]$. Recall that $\nu(G)=s+1$. We know that $\nu(G[V_2])=2$. It follows that $|M|\geq 2$.

\textbf{Subcase 3.1.} $|M|=2$.

Then there exist a color $c$ and an edge set $E_1\subseteq E(H[V_2])$ with $f(e)=c$ for all $e\in E_1$ such that $\nu(H[V_2] \setminus E_1) = 1$. Let $E_2:=E(H[V_2])-E_1$. 
We know that $H[E_2]$ is a triangle or a star. 
Note that
\begin{align*}
	e(G[V_2]) &\geq  g_1(n, s)-\binom{s-1}{2} - (s-1)(n-s+1)\\
 &\geq \binom{n}{2} - \binom{n-s+1}{2} + 5- \binom{s-1}{2} - (s-1)(n-s+1) =  5,
\end{align*}
that is
\begin{align}\label{case3-eq1}
	e(G[V_2]) > 5.
\end{align}
If $H[E_2]$ is a triangle, then  $|\{f(e)| \ e\in E(H[V_2])\}|=4$, which contradicts  (\ref{case3-eq1}). Thus we may assume that $H[E_2]$ is star with at least $5$ vertices. 
Since $G[V_2]$ is a subgraph of $K_2\vee \overline{K_{n-s-1}}$ and $e(G[V_2])\geq 5$, by (\ref{order}), it follows that $e_G(\{v_1\}, V(G-T))\geq 1$ and $e_G(\{v_2\},V(G-T))\geq 3$. 
Therefore, $H[V_2]-v_2$ is a  monochromatic  copy of  $K_{n-s}$  with color $c$ in $H$.

\medskip
\textbf{Subcase 3.2.~}$|M|\geq 3$.

Then, we can choose a rainbow matching $\mathcal{M}$ of size 3 in $H[V_2]$ such that $V(\mathcal{M})\cap \{v_1,v_2\}\neq \emptyset$. 
We define $\mathcal{C}:=\{f(e)\ |\ e\in \mathcal{M}\}$ and  $E_1:=\{e\in E(G)\ |\ f(e)\in \mathcal{C}\}$. By the definition of $G$, 
we know that $|E_1|=3$. Let $A:=\{v_3,v_4,\ldots,v_{s+1}\}$ and $B:=\{v_{s+2}, v_{s+3},\ldots, v_n\}-V(\mathcal{M})$.  Let $G'$ be a bipartite graph with bipartition $(A,B)$ and edge set  $E_G(A,B)-E_1$.

\medskip
\textbf{Claim 5.~}$e_{G'}(v_i, B)\geq i - 2$ for all $v_i\in A$.

 Suppose, for contradiction, that there exists $v_j\in A$ such that $e_{G'}(v_j, B)\leq j - 3$. Recall that $|V(M)\cap \{v_1,v_2\}|\neq \emptyset$, so $e_{G-E_1}(v_j, B)\leq j+2$. 
 Since $|E_1|= 3$, by (\ref{order}), we have
\begin{align*}
	\sum_{k=1}^je_G(\{v_k\}, V(G-T))\leq \sum_{k=1}^je_{G-E_1}(\{v_k\}, V(G-T)) +3&\leq j(j+2) + 3.
\end{align*}
It follows that
\begin{align*}
	 e(G)&\leq \sum_{k=1}^je_G(\{v_k\}, V(G-T))+ \binom{s+1}{2} + (s + 1 - j)(n-s-1)\\
&\leq \binom{s+1}{2} + (s + 1 - j)(n-s-1)+ j(j+2) + 3.
\end{align*}
Let $h_2(n, s, j):= \binom{s+1}{2} + (s + 1 - j)(n-s-1)+ j(j+2) + 3.$
The function $h(n, s, j)$ is convex for $j$ in the interval $[3, s+1]$, and so $h_2(n, s, j)\leq \max\{h_2(n, s, 3), h_2(n, s, s+1)\}$. 
Given $n\geq \{2s+5, 40\}$,  we have
\begin{align*}
h_2(n, s, 3)&=\binom{s+1}{2} + (s - 2)(n-s-1) + 18\\
&=ns- 2n - \frac{1}{2}s^2 + \frac{3}{2}s + 20<g_1(n,s),
\end{align*}
and
\begin{align*}
	h_2(n, s, s+1)&=\binom{s+1}{2} + (s+1)(s+3) + 3\\
	&=\frac{3}{2}s^2 + \frac{9}{2}s + 6<g_2(n,s).
\end{align*}
Thus
$e(G)<  \max \{g_1(n,s), g_2(n,s)\}$,  a contradiction. This completes the proof of Claim 5.

 By Lemma \ref{lemma} and Claim 5,  $G'$ contains a matching $M'$ of size $s-1$. Then $M \cup M'$ forms  a rainbow matching of size $s+2$ in $H$,  a contradiction. This completes the proof of Case 3.

\medskip
\textbf{Case 4.} $t=s$.

According to Claim 3, 
$G$  is a subgraph of the graph $
K_{s} \vee (K_3\cup\overline{K_{n-s-3}}).
$ 
Since the components of $G-T$ are factor-critical and $\nu(G)=s+1$, one can see that $G-T\cong K_3\cup \overline{K_{n-s-3}}$. Note that
\begin{align*}
e_G(\{v_1\}, V(G-T)) &\geq e(G)- (s-1)(n-s)-{s\choose 2}-3 \\
&\geq \binom{n}{2} - \binom{n-s+1}{2} + 5 - \binom{s}{2} - (s-1)(n-s)-3 =  2.
\end{align*}
Let $V_3:=\{v_1\}\cup \{v_{s+1}, v_{s+2}, \ldots, v_n\}$. Define $G':=G[V_3]$.  Note that $e(G')\geq 5$ and $G'$ contains a triangle. Thus  for any 
$e\in E(G')$, $G'-e$ contains a matching  of size $2$  covering $v_1$. 
 It follows that $H[V_3]$ has a rainbow matching $\mathcal{M}$ of size $3$ such that $v_1\in V(\mathcal{M})$ and $|E(G-T) \cap \mathcal{M}|=1$.   Without loss generality, we may assume that $\mathcal{M}:=\{e_{1u}, e_{s+1,s+2}, e_{n-1,n}\}$ and $e_{1u}\in E(G)$.
Let $e'\in E(G)$ such that $f(e')=f(e_{n-1,n})$.
Let $A=T-v_1$ and $B=V(G-T) - V(M)$. Let $F$ be a bipartite graph with partition $V(F)=A\cup B$, whose edge set is $E_G(A,B)-e'$.
 Then we have that for all $v_i\in A$,
\begin{align}\label{differ1}
	e_G(\{v_i\}, V(G-T)) \leq e_{F}(\{v_i\}, B) + 6.
\end{align}


\medskip
\textbf{Claim 6.~}$e_{F}(\{v_i\}, B)\geq i - 1$ for all $v_i\in A$.

 Suppose, for contradiction, that there exists $v_j\in A$ such that $e_{F}(\{v_j\}, B)\leq j - 2$. By (\ref{order}) and (\ref{differ1}), we can deduce that $e_{G}(\{v_k\}, V(G-T))\leq j + 4$ for all $k\in [j]$. Thus, we have
\begin{align*}
	\sum_{k=1}^je_G(\{v_k\}, V(G-T)) &\leq j(j+4).
\end{align*}
It follows that
\begin{align*}
	e(G)& \leq \sum_{k=1}^je_G(\{v_k\}, V(G-T))+\sum_{k=j+1}^se_G(\{v_k\}, V(G-T))+3\\
&\leq \binom{s}{2} + (s-j)(n-s) + j(j+4) + 3.
\end{align*}
Let $h_3(n,s,j):=\binom{s}{2} + (s-j)(n-s) + j(j+4) + 3$. Then $h_3(n,s,j)$
 is   convex for $j$  in the interval $[2,s]$. Thus, $h_3(n,s,j)\leq \max \{h(n,s,2), h(n,s,s)\}$. Then we have
\begin{align*}
	h_3(n,s,2)= \binom{s}{2} + (s-2)(n-s) + 15< g_1(n,s),
\end{align*}
and
\begin{align*}
	h_3(n,s,s) =\binom{s}{2} +  s(s+4) + 3 <g_2(n, s).
\end{align*}
 Consequently, $e(G)\leq \max \{h_3(n,s,2), h_3(n,s,s)\} < \max \{g_1(n,s), g_2(n,s)\}$,  a contradiction. This completes the proof of Claim 6.

%

 By Lemma \ref{lemma} and Claim 6, $F$ has  a matching $M'$ of size $s-1$. Then $M \cup M'$ forms a rainbow matching of size $s+2$ in $H$,  a contradiction. This completes the proof of Case 4.

\medskip
\textbf{Case 5.} $t\in \{2,3\}$.

According to Claim 3,
$G$  is a subgraph of $K_t \vee (K_{2s-2t+3} \cup \overline{K_{n-2s+t-3}}).$ Recall that $V(C_1)=\{v_{t+1}, v_{t+2},\ldots,v_{2s-t+3}\}$.
Let $H':=H-V(C_1)$. 

\medskip
\textbf{Claim 7.~} $H'$ contains a rainbow matching of size $t$.

Suppose, for the sake of contradiction, that this is not the case.
 
When $t = 2$, $H'$ is a monochromatic graph. Let $f(e)=c$ for all $e\in E(H')$. We aim to show there is a monochromatic $K_{n-2s-1}\vee \overline{K_{2s+1}}$ in $H$. It suffices to show that $f(e_{uv})=c$ for all edges $e_{uv}$ with $u\in [3, 2s+1]$ and $v\in [2s+2, n]$. Otherwise, suppose there exist $u\in [3, 2s+1]$ and $v\in [2s+2, n]$ such that $f(e_{uv})\neq c$. Let $e_1\in E(G)$ such that $f(e_1)=f(e_{uv})$. We select an edge $e_2=e_{i,j}$ with $\{i,j\}\subseteq [2s+2,n]-\{v\}$, $f(e_2)=c$, and an edge $e_3\in E(G)$ with $f(e_3)=c$. Let $V_4:=\{v_1,v_2,\ldots, v_{2s+1}\}$. Then we have 
\begin{align*}
	&e(G[V_4]-\{v_u\}-\{e_1, e_3\})\\
	&\geq e(G[V_4])-2s-2\\
	&\geq g_2(n,s)-2s-2\\
	&>\binom{2s-1}{2}.
\end{align*}
Hence, by Theorem \ref{Matching}, $e(G[V_4]-\{v_u\}-\{e_1, e_3\})$ has a matching $M$ of size $s$. Then $M\cup \{e_{uv}, e_2\}$ forms a rainbow matching of size $s+2$ in $H$, a contradiction.

When $t = 3$, by Theorem \ref{anti-matching}, $H'$ contains at most $n - 2s + 3$ colors. In this situation, 
 the number of edges in $G$ satisfies 
 \[
 e(G)\leq \binom{2s - 3}{2}+ n - 2s + 3 + 3(2s-3)=n+2s^2-3s<g_2(n,s),
 \]
a contradiction. This completes the proof of Claim 7.

Consider that
 $H'$ contains a rainbow matching $\mathcal{M}_1$ of size $t+1$. Define $\mathcal{C}_1:=\{f(e)\ |\ e\in \mathcal{M}_1\}$ and $E_1:=\{e\in E(G)\ |\ f(e)\in \mathcal{C}_1\}$.
Since $H$ contains no rainbow matching of size $s+2$, then we know that $\nu(C_1-E_1)\leq s-t$.  Then, by Theorem  \ref{Matching}, we have
\begin{align*}
e(C_1)-(t+1)\leq \max\{{2s-2t+1\choose 2},(s-t)(s-t+3)+{s-t\choose 2}\}.
\end{align*} It follows that
\begin{align*}
e(G)&\leq e(C_1)+t(n-t)+{t\choose 2}\\
&\leq \max\{{2s-2t+1\choose 2},(s-t)(s-t+3)+{s-t\choose 2}\}+(t+1)+t(n-t)+{t\choose 2}.
\end{align*}
Define
\[
p_1(n,s,t):=\max\{{2s-2t+1\choose 2},(s-t)(s-t+3)+{s-t\choose 2}\}+(t+1)+t(n-t)+{t\choose 2}.
\]
If $s-t\geq 3$, then we get
\[
p_1(n,s,t):={2s-2t+1\choose 2}+(t+1)+t(n-t)+{t\choose 2}<\max\{g_1(n,s),g_2(n,s)\},
\]
 a contradiction.
Otherwise, $s-t=2$. Since $n\geq \{2s+5, 40\}$, then we have
\[
p_1(n,s,t):=12+t+t(n-t)+{t\choose 2}<g_1(n,s),
\]
a contradiction.

Next, we assume that $H'$ contains no  rainbow matching  of size $t+1$. Recall that $n\geq 2s+5. $ 
 By Theorem \ref{anti-matching}, $H'$ has at most $(t-1)(n-2s+t-2)+{t-1\choose 2}+1$ colors.
By Claim 7, we may select a rainbow matching $\mathcal{M}_2$ of size $t$ in $H'$ such that $T-V(\mathcal{M}_2)\neq \emptyset $. Let $v_u\in T-V(\mathcal{M}_2)$ and define $\mathcal{C}_2:=\{f(e)\ |\  e\in \mathcal{M}_2\}$. Also, let $\mathcal{M}_2':=\{e\in E(C_1)\ |\ f(e)\in \mathcal{C}_2\}$.
Then we have $\nu(G[V(C_1)\cup \{v_u\}]-\mathcal{M}_2')\leq s-t+1$.   So, by Theorem \ref{Matching}, we get
\[
e(G[V(C_1)\cup \{v_u\}])\leq {2s-2t+3\choose 2}+|\mathcal{M}_2'|.
\]
It follows that
\begin{align*}
 e(G)&\leq e(G[V(C_1)\cup \{v_u\}])+(t-1)(2s-2t+3)+(t-1)(n-2s+t-2)+{t-1\choose 2}+1-|\mathcal{M}_2'|\\
&\leq {2s-2t+3\choose 2}+(t-1)(2s-2t+3)+(t-1)(n-2s+t-2)+{t-1\choose 2}+1\\
&={2s-2t+3\choose 2}+(t-1)(n-t+1)+{t-1\choose 2}+1.
\end{align*}
Let
\[
p_2(n,s,t):={2s-2t+3\choose 2}+(t-1)(n-t+1)+{t-1\choose 2}+1.
\] Then $p_2(n,s,2)=g_2(n,s)-1$ and
\[
p_2(n,s,3)={2s-3\choose 2}+2n-2<\max\{g_1(n,s),g_2(n,s)\}.
\]
So for $t\in \{2,3\}$, $e(G)<p_2(n,s,t)<g(n,s)$, a contradiction. This completes the proof. \qed

\end{document}